\documentclass[letterpaper,11pt]{amsart}
\usepackage[latin1]{inputenc}
\theoremstyle{plain}
\newtheorem{theorem}{Theorem}

\theoremstyle{definition}

\newtheorem*{example}{Example}

\theoremstyle{remark}
\newtheorem*{remark}{Remark}

\usepackage{latexsym}
\usepackage{graphicx}
\usepackage{setspace}
\usepackage{mathtools}
\usepackage{multicol}
\usepackage{doi}
\numberwithin{equation}{section}

\newcommand\xqed[1]{%
	\leavevmode\unskip\penalty9999 \hbox{}\nobreak\hfill
	\quad\hbox{#1}}
\newcommand\demo{\xqed{$\triangle$}}

\newcommand{\transp}[1]{{#1}^\top}
\newcommand{\cda}[1]{\mathbb{A}_{#1}}
\usepackage{pst-qtree}

\author{John W. Bales}
\address{Department of Mathematics(Retired)\\Tuskegee University\\Tuskegee, AL  36088}
\email{john.w.bales@gmail.com}
\curraddr{PO Box 210, Waverly, AL 36879}
\title{Periodicity of the Cayley-Dickson twists}
\date{}
\begin{document}
\begin{abstract}
 Regarding the Cayley-Dickson algebras as twisted group algebras, this paper identifies the periodic character of these twists.
\end{abstract}
\maketitle

\section{Introduction}
The unit basis vectors $\{e_k\}$ of Cayley-Dickson algebras may be represented as a \emph{twisted group} with $e_0$ as the group identity. For each of the eight Cayley-Dickson doubling products \cite{B2016} there is a twisting map $\omega(p,q):\mathbb{N}_0^2\mapsto\{\pm1\}$ (where $\mathbb{N}_0$ represents the non-negative integers) with the property that $e_pe_q=\omega(p,q)e_{p\oplus q}$ where $\oplus$ is a group operation on $\mathbb{N}_0$ consisting of the `bit-wise exclusive or' of the binary representations of non-negative integers.

This paper identifies the periodic nature common to all eight of these twisting maps.
\section{Background}

Cayley-Dickson algebras are here regarded as twisted group algebras \cite{BS1970,R1971} on $2^N$ dimensional Euclidean subspaces of the Hilbert space $\ell^p$ of square-summable sequences using the standard unit basis $e_0=1,0,0,0,\cdots$, $e_1=0,1,0,0,0,\cdots$ etc. together with the standard norm and inner product. Since Cayley-Dickson algebras exist in a sequence $\{\cda{k}\}$ where the algebra $\cda{N+1}$ consists of all ordered pairs $\{(a,b)\vert a,b\in\cda{N}\}$, an ordered pair $(a,b)$ will be regarded as the `shuffle' of sequence $a=a_0,a_1,a_2,\cdots$ and sequence $b=b_0,b_1,b_2,\cdots$ so that $(a,b)=a_0,b_0,a_1,b_1,\cdots$. A real number $x$ is identified with the sequence $x,0,0,0,\cdots$ so that $\cda{0}=\mathbb{R}.$ The \emph{conjugate} $x^*$ of a sequence $x=x_0,x_1,x_2,\cdots$ satisfies $x+x^*\in\mathbb{R}$ thus we define $x^*=x_0,-x_1,-x_2,\cdots$. Equivalently, 
\begin{equation}
(a,b)^*=(a^*,-b)
\end{equation}
 This approach generates the `shuffle basis' on the infinite dimensional Cayley-Dickson algebra $\cda{}$ where the unit basis vectors are defined recursively as
\begin{align}
  e_0&=1\\
  e_{2k}&=(e_k,0)\text{\ for\ }k\ge0\\
  e_{2k+1}&=(0,e_k)\text{\ for\ }k\ge0
\end{align}

The algebra $\cda{}=\bigcup\cda{k}$ contains as proper subalgebras the real numbers $\cda{0}$, the complex numbers $\cda{1}$, quaternions $\cda{2}$, octonions $\cda{3}$, sedenions $\cda{4}$, etc. In general, if $a,b\in\cda{}$ then so is the ordered pair $(a,b)$. Furthermore, for each $n\ge0$, $\cda{n}\subset\cda{n+1}$.

There exist eight basic Cayley-Dickson \emph{doubling products} \cite{B2016} which are listed in Table \ref{tab:EightProducts}.
\begin{table}[ht]
$P_{0}\,\,:   (a,b)(c,d)=(ca-b^*d,da^*+bc)$

$P_{1}\,\,:   (a,b)(c,d)=(ca-db^*,a^*d+cb)$
  
$P_{2}\,\,:   (a,b)(c,d)=(ac-b^*d,da^*+bc)$

$P_{3}\,\,:   (a,b)(c,d)=(ac-db^*,a^*d+cb)$ 

$\transp{P}_{0}:  (a,b)(c,d)=(ca-bd^*,ad+c^*b)$

$\transp{P}_{1}:  (a,b)(c,d)=(ca-d^*b,da+bc^*)$

$\transp{P}_{2}:  (a,b)(c,d)=(ac-bd^*,ad+c^*b)$

$\transp{P}_{3}:  (a,b)(c,d)=(ac-d^*b,da+bc^*)$\\
~
\caption{The Eight Cayley-Dickson Doubling Products}\label{tab:EightProducts}
\end{table}

A search of the literature on Cayley-Dickson algebras reveals the use of only the two doubling products $P_{3}$ and $\transp{P}_{3}$. In general for a doubling product $P_k$, $e_pe_q=e_r$ if and only if for the doubling product $\transp{P}_k$ it is the case that $e_qe_p=e_r$. The basis product tables for $P_k$ and $\transp{P}_k$ are each the transpose of the other. Note, however, that this realtionship holds only for the basis vectors. If $xy=z$ in $P_{k}$ it does not follow that $yx=z$ for $\transp{P}_{k}$. Attention will be restricted in this paper to the first four of the products.

The proofs of the theorems involve the use of Cayley-Dickson `twist trees.' \cite{B2009} The next two sections discuss such twists and their trees.

\section{Cayley-Dickson Twists}

For each of the eight Cayley-Dickson doubling products there is a unique \emph{twist} function $\omega:\mathbb{N}_0\times \mathbb{N}_0\mapsto\{-1,1\}$ such that for $p,q\in \mathbb{N}_0$
\begin{equation}
	e_pe_q=\omega(p,q)e_{p\oplus q}
\end{equation}
where $p\oplus q$ is the bit-wise `exclusive or' of the binary representations of $p$ and $q$. This relationship among the basis vectors is a natural result of regarding ordered pairs of sequences as the shuffle of the two sequences and is equivalent to addition in $\mathbb{Z}_2^N$.

 To illustrate
 , \[5\oplus 11=0101_B\oplus1011_B=1110_B=14\] 
 so 
 \[e_5e_{11}=\omega(5,11)e_{14}\]
 
\section{Navigating the Twist Trees}

Cayley-Dickson twist trees \cite{B2009} are abstracted from observations about the structure of the multiplication table of the unit vectors $e_k$ and are illustrated in Figures \ref{fig:quaternionTree}, \ref{fig:generalTree} and \ref{fig:fourITrees}. 

To find the value of $\omega(p,q)$ in a basis vector product $e_pe_q=\omega(p,q)e_{p\oplus q}$ one needs a set of navigation instructions for the Cayley-Dickson twist tree. This set of instructions is symbolized by the bracketed ordered pair $[p;q]$ and details how to `navigate' the tree by following a sequence of left-right instructions beginning at the root node. After navigating the tree according to the instructions $[p;q]$ the sign of the terminal node will be the value of $\omega(p,q).$

The process of converting the symbol $[p;q]$ into a set of left-right navigation instructions for the tree is as follows:

\begin{enumerate}
	\item Convert $[p;q]$ to the binary representations of $p$ and $q$ padding the smaller of the two with leading 0's when necessary to maintain an equal number of bits. Example: $[26;42]=[011010;101010]$.
	\item Shuffle the two binary strings into binary doublets. $[011010;101010]=01,10,11,00,11,00$
	\item Interpret a 0 as an instruction to move down the left branch (L) from the current node and a 1 as an instruction to move down the right branch (R) from the current node. Example $01,10,11,00,11,00\Rightarrow LR,RL,RR,LL,RR,LL$.
\end{enumerate}
The bracket notation $[p;q]$ denoting the shuffle of \emph{binary numbers} $p,\,q$ is used to avoid confusion with the parenthesis notation denoting the shuffle $(x,y)$ of \emph{number sequences} $x,\,y$.
\begin{example}
Refer for this example to the quaternion tree in Figure \ref{fig:quaternionTree}. The product $e_3e_1=\omega_3(3,1)e_{3\oplus1}=\omega(3,1)e_2$. The integers 3 and 1 are shuffled by pairing the bits of $3$ with the bits of $1$:  $[3;1]=[11;01]=10,11$. 

This string of bits is taken as the navigation instruction RL,RR for the twist tree. Following the instructions leads to the twelvth terminal node from the left which is labled with $1$. Thus $\omega(3,1)=1$. So $e_3e_1=e_2$. \demo
\end{example}
If we identify $e_1$, $e_2$ and $e_3$ with quaternions $i$, $j$ and $k$, respectively, this gives $ki=j$. This particular twist tree is identical for all four of the products $P_0$ through $P_3$ and produces the correct products for the quaternion basis vectors.

\begin{figure}[ht]
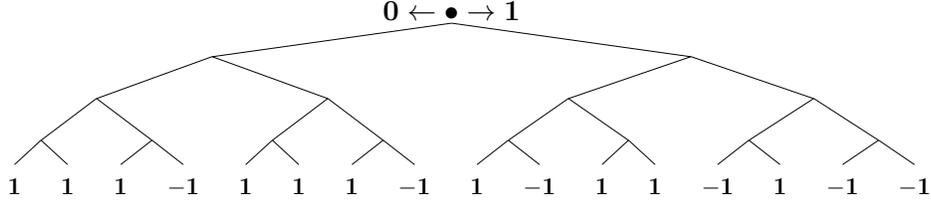

\centering
\scalebox{0.44}{
\psset{nodesep=-0.05}
\Tree[.{\larger[5]$\mathbf{0\leftarrow\bullet\rightarrow1}$}
           [.{} 
            [.{} 
             [.{} 
              [.\\{\huge$\mathbf{1}$} 
              ]
              [.\\{\huge$\mathbf{1}$} 
              ]
             ]
             [.{} 
              [.\\{\huge$\mathbf{1}$} 
              ]
              [.\\{\huge$\mathbf{-1}$} 
              ]
             ]
            ]
            [.{} 
             [.{} 
              [.\\{\huge$\mathbf{1}$} 
              ]
              [.\\{\huge$\mathbf{1}$} 
              ]
             ]
             [.{} 
              [.\\{\huge$\mathbf{1}$} 
              ]
              [.\\{\huge$\mathbf{-1}$} 
              ]
             ]
            ]
           ]
           [.{} 
            [.{} 
             [.{} 
              [.\\{\huge$\mathbf{1}$} 
              ]
              [.\\{\huge$\mathbf{-1}$} 
              ]
             ]
             [.{} 
              [.\\{\huge$\mathbf{1}$} 
              ]
              [.\\{\huge$\mathbf{1}$} 
              ]
             ]
            ]
            [.{} 
             [.{} 
              [.\\{\huge$\mathbf{-1}$} 
              ]
               [.\\{\huge$\mathbf{1}$} 
               ]
               ]
              [.{} 
               [.\\{\huge$\mathbf{-1}$} 
               ]
               [.\\{\huge$\mathbf{-1}$} 
               ]
              ]
             ]
            ]
     ]
}
\caption{Quaternion Twist Tree $P_0$ through $P_3$}
\label{fig:quaternionTree}
\end{figure}

\section{Using the general tree to calculate $e_pe_q$}

A version of the general quaternion tree is developed in \cite{B2016}. The general tree depicted in Figures \ref{fig:generalTree} and \ref{fig:fourITrees} suffices for Cayley-Dickson algebras of any dimension for products $P_0$ through $P_3$. A variation is valid for the transposes of these four products.

The meaning of the letters C, L, T, D and I are explained in \cite{B2009} (although different letters were used there) but here it will suffice to say that they represent corner (C), left (L), top (T), diagonal (D) and interior (I) as illustrated in Figure \ref{fig:cdMatrix} referencing sections of the product table of the basis vectors, specifically with regard to the values of $\omega(p,q)$ in those tables. Additionally, Figure \ref{fig:cdMatrix} indicates how an $\omega$ table for $\cda{N}$ transitions into an $\omega$ table for $\cda{N+1}$. For the complex numbers $\mathbb{C}=\cda{1}$, $\text{C}=\text{T}=\text{L}=\text{D}=\text{I}=1$. The matrix on the right shows the twist matrix for the quaternions $\mathbb{Q}=\cda{2}$. Applying the transformation one more time would show the twist matrix for the octionions $\mathbb{O}=\cda{3}$. The Cayley-Dickson tree for the eight doubling products differ only in the behavior of their interior (I) nodes, requiring a separate I-tree for each of the eight.

\begin{figure}[ht]
{\setstretch{2.25}
\begin{align*}
	\left(
	\begin{array}{r|r}
	C &  T\\
	\hline
	L   & -D\\
	\end{array}\right) &\Longrightarrow
			     \left(\begin{array}{rr|rr}
			     	C &  T &  T &  T\\
			     	L   & -D &  I & -I\\
			     	\hline
			     	L   & -I & -D &  I\\
			     	L   &  I & -I & -D
			     \end{array}\right)
\end{align*}
}
\caption{Matrix version of Figure \ref{fig:generalTree}}
\label{fig:cdMatrix}
\end{figure}

The twist tree in Figure \ref{fig:generalTree} is valid for all four doubling products $P_0$ through $P_3$ but the trees for I shown in Figure \ref{fig:fourITrees} vary.

\psset{nodesep=-0.05}
\begin{figure}[b]
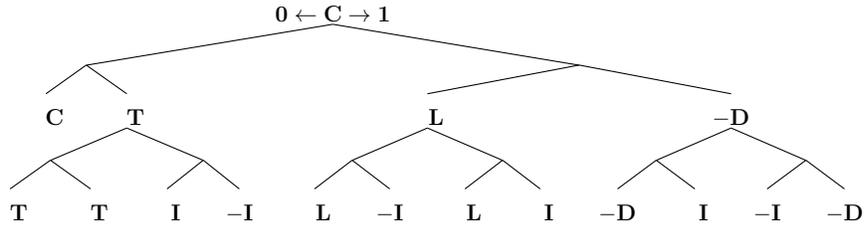

	\scalebox{0.5}{
		\Tree [.{\larger[3]$\mathbf{0\leftarrow C\rightarrow1}$} 
		[.{}  
		[.\\{\larger[3]$\mathbf{\phantom{-}C }$} 
		] 
		[.\\{\larger[3]$\mathbf{\phantom{-}T}$} 
		[.{} 
		[.\\{\larger[3]$\mathbf{\phantom{-}T}$} 
		] 
		[.\\{\larger[3]$\mathbf{\phantom{-}T}$} 
		] 
		] 
		[.{} 
		[.\\{\larger[3]$\mathbf{\phantom{-}I}$} 
		] 
		[.\\{\larger[3]$\mathbf{-I}$} 
		] 
		] 
		] 
		] 
		[.{} 
		[.\\{\larger[3]$\mathbf{\phantom{-}L}$} 
		[.{} 
		[.\\{\larger[3]$\mathbf{\phantom{-}L}$} 
		] 
		[.\\{\larger[3]$\mathbf{-I}$} 
		] 
		] 
		[.{} 
		[.\\{\larger[3]$\mathbf{\phantom{-}L}$} 
		] 
		[.\\{\larger[3]$\mathbf{\phantom{-}I}$} 
		] 
		] 
		] 
		[.\\{\larger[3]$\mathbf{-D}$} 
		[.{} 
		[.\\{\larger[3]$\mathbf{-D}$} 
		] 
		[.\\{\larger[3]$\mathbf{\phantom{-}I}$} 
		] 
		] 
		[.{} 
		[.\\{\larger[3]$\mathbf{-I}$} 
		] 
		[.\\{\larger[3]$\mathbf{-D}$} 
		] 
		] 
		] 
		] 
		] 
	}
	\caption{Twist tree for $\omega_0$ through $\omega_3$}
	\label{fig:generalTree}
\end{figure}

\begin{figure}[b]
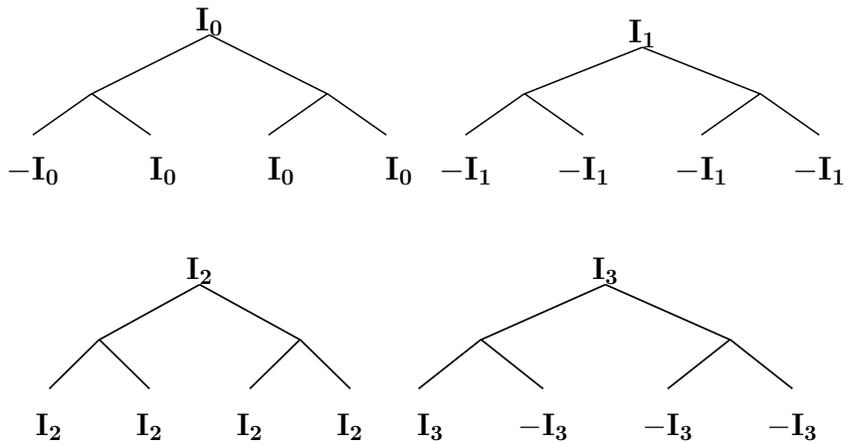

	\centering
	\scalebox{0.72}{
		\Tree[.{\larger[3]$\mathbf{I_0}$}
		[.{} 
		[.\\{\larger[3]$-\mathbf{I_0}$} 
		] 
		[.\\{\larger[3]$\phantom{-}\mathbf{I_0}$} 
		] 
		] 
		[.{} 
		[.\\{\larger[3]$\phantom{-}\mathbf{I_0}$} 
		] 
		[.\\{\larger[3]$\phantom{-}\mathbf{I_0}$} 
		] 
		] 
		]
		
		\Tree[.\\{\larger[3]$\mathbf{I_1}$}
		[.{} 
		[.\\{\larger[3]$-\mathbf{I_1}$} 
		] 
		[.\\{\larger[3]$-\mathbf{I_1}$} 
		] 
		] 
		[.{} 
		[.\\{\larger[3]$-\mathbf{I_1}$} 
		] 
		[.\\{\larger[3]$-\mathbf{I_1}$} 
		] 
		] 
		]}\\
	
	\scalebox{0.82}{
		\Tree[.\\{\larger[2]$\mathbf{I_2}$}
		[.{} 
		[.\\{\larger[2]$\mathbf{I_2}$} 
		] 
		[.\\{\larger[2]$\mathbf{I_2}$} 
		] 
		] 
		[.{} 
		[.\\{\larger[2]$\mathbf{I_2}$} 
		] 
		[.\\{\larger[2]$\mathbf{I_2}$} 
		] 
		] 
		]
		
		\Tree[.\\{\larger[2]$\mathbf{I_3}$}
		[.{} 
		[.\\{\larger[2]$\phantom{-}\mathbf{I_3}$} 
		] 
		[.\\{\larger[2]$-\mathbf{I_3}$} 
		] 
		] 
		[.{} 
		[.\\{\larger[2]$-\mathbf{I_3}$} 
		] 
		[.\\{\larger[2]$-\mathbf{I_3}$} 
		] 
		] 
		]
	}       
	
	\caption{Twist trees for interior points of $\mathbb{N}_0\times \mathbb{N}_0$}
	\label{fig:fourITrees}
\end{figure}

The following example illustrates how to use the general tree to find the product of two basis vectors using the product $P_3$ from Table \ref{tab:EightProducts}.

\begin{example}
	Let us calculate the product $e_{25}e_{17}$.
	
	First find 
	\[25\oplus17=1101_B\oplus1001_B=0100_B=8\] 
	Then 
	\[e_{25}e_{17}=\omega_3(25,17)e_8\] 
	To find the tree navigation instructions shuffle the bits of the binary representations
	\[ [25;17]=[1101;1001]=11,10,00,11(\Rightarrow RR,RL,LL,RR)\]
	In Figure \ref{fig:generalTree}, $11,10$ takes us first to $-$D then to $-$I. 
	
	Using the tree for $\text{I}_3$ in Table \ref{fig:fourITrees} but \emph{reversing all the signs} since it's actually$-\text{I}_3 $ in this case, the $00$ takes us from $-\text{I}_3$ to to $-\text{I}_3$ and the $11$ takes us from $-\text{I}_3$ to $+\text{I}_3$. The sign of the final node is always the sign of $\omega(p,q)$. 
	
	Therefore, $\omega(25,17)=1$.
	
	Thus $e_{25}e_{17}=e_8$.\demo
\end{example}

\section{Periodic properties of $\omega$}
second 
Because of the periodicities of $\omega$, computing $\omega$ for large $p,\,q$ can be simplified by computing $\omega(p^\prime,q^\prime)$ for smaller $p^\prime,\,q^\prime$ having the same $\omega$ value. This is useful, for example, in algorithms for computing products of vectors in higher dimensional Cayley-Dickson algebras. Investigation of applications of Cayley-Dickson algebras in areas such as signal processing, file compression or cryptography involving, for example $\cda{10}$ would require $2^{20}$ computations of $\omega(p,q)$ when multiplying a single pair of vectors $x$ and $y$.

\begin{theorem}
The twists of all eight products satisfy the following: \cite{B2016}
\begin{align*}
 \omega(p,0)&=\omega(0,p)=1\text{ for all }p.\\
 \omega(p,p)&=-1\text{ for all }p>0.\\
 \omega(p,q)&=-\omega(q,p)\text{ for }0\ne p\ne q\ne 0
\end{align*}
\label{Thm:known}
\end{theorem}

The properties in Theorem \ref{Thm:known} are well known but not theorem \ref{Thm:A}, which can be proved using the trees in Figures \ref{fig:generalTree} and \ref{fig:fourITrees}.

\begin{remark}
	For a positive integer $p$ the inequality $2^{N-1}\le p<2^N$ means that $p$ cannot be expressed in binary form with fewer than $N$ bits.
\end{remark}

\begin{theorem} The first periodicity theorem:
If $2^{N-1}\le p<2^{N}\le q<2^{N+1}$ and $k\ge0$ then
\begin{equation*}
 \omega(p,q)=\omega(p,q+k2^N)
\end{equation*}
\label{Thm:A}
\end{theorem}
\begin{proof}
Suppose $2^{N-1}\le p<2^{N}\le q<2^{N+1}$. 

Then the two highest order bits of $p$ are 01 and the two highest order bits of $q$ are $1x$ where $x\in\{0,1\}$. 

So $[p;q]=[01\cdots;1x\cdots]=01,1x,\cdots$. Thus, before applying the doublet $1x$ the current node is T. 

Now consider $[p;q+k2^N]=01,\cdots,1x,\cdots$.  The first ellipsis can contain only 00 or 01 so before applying the doublet $1x$ the node is also T.  

Yet in both $[p;q]=01,1x,\cdots$ and $[p;q+k2^N]=01,\cdots,1x,\cdots$ the final ellipses are identical.  Thus $\omega(p,q)=\omega(p,q+k2^N)$. 
\end{proof}

\begin{theorem}\label{Thm:B}The second periodicity theorem:
 If $2^{N-1}\le p<2^N$ and $2^{N-1}\le q<2^N$ and $k\ge0$ then
 \begin{equation*} \omega(p,q)=\omega(p+k2^N,q+k2^N)
 \end{equation*}
\end{theorem}
\begin{proof}
Here the path $[p;q]=[1\cdots;1\cdots]=11,\cdots$ begins at $-$D.

In the path $[p+k2^N;q+k2^N]=[1\cdots1\cdots;1\cdots1\cdots]=11,\cdots,11,\cdots$ the first instructions brings one to $-$D and the first ellipsis consists entirely of 00s or 11s. Thus the path remains at $-$D until reaching the final ellipsis. But in paths $[p;q]$ and $[p+k2^N;q+k2^N]$, the final ellipses are the same. 

Thus $\omega(p,q)=\omega(p+k2^N,q+k2^N)$.
\end{proof}
\section{The modularity properties}

Using modular arithmetic, properties related to those in Theorems \ref{Thm:A} and \ref{Thm:B} may be stated.

\begin{theorem}
 If $2^{N-1}\le p<2^{N}\le q$ then 
 
 \begin{equation*}
 \omega(p,q)=\omega(p,2^N+q\bmod{2^N})
 \end{equation*}
\end{theorem}
 \begin{proof}
 Suppose $2^{M-1}\le q<2^M$ where $N\le M$. Represent $p$ by the the binary string $0\cdots 1\cdots$ where the $1$ is the $N$th bit from the right and $0$ the $M$th. Represent $q$ by the binary string $1\cdots x\cdots$ where $x$ is the $N$th bit from the right and $1$ the $M$th. Then $[p;q]=01,\cdots,1x,\cdots$. In the tree diagram in Figure \ref{fig:generalTree}, on page \ref{fig:generalTree} the instruction $01$ brings us to $T$. The first ellipsis in $[p;q]$ is either null or consists of only $00$ or $01$, either of which leaves one at $T$. So the value of $\omega(p,q)$ will be the same as it would if the first ellipsis were empty. If $q$, with binary representation $1\cdots x\cdots$ were to be replaced with $q^\prime$ with binary string $1x\cdots$ where the rightmost ellipsis of $q$ is the same binary string as the ellipsis in $q^\prime$, then $\omega(p,q)=\omega(p,q^\prime).$ So $q$ may be replaced by $q^\prime=2^N+q\bmod{2^N}$ and the value of $\omega$ will be the same.
\end{proof}
\begin{example} 
	$\omega(5,481)=\omega(5,2^3+481\bmod{2^3})=\omega(5,9)$	\demo
\end{example}
\begin{theorem}
Suppose $2^{M-1}\le p<2^{M}$, $2^{M-1}\le q<2^{M}$ and $2^{N-1}\le p\oplus q<2^{N}$. Then $N<M$ and
 \begin{equation*}
  \omega(p,q)=\omega(2^N+p\bmod{2^N},2^N+q\bmod{2^N})
 \end{equation*}
\end{theorem}
\begin{proof}
Represent $p=1\cdots x\cdots$ and $q=1\cdots y\cdots$ with both $1$s the $M$th bit from the right and with $x$ and $y$ the $N$th bits from the right. Then $[p;q]=11,\cdots xy,\cdots$ with the first ellipsis being empty or consisting of only $11$ or $00$. So applying these navigating instructions for the tree in Figure \ref{fig:generalTree} places one at $-D$ when arriving at the instruction $xy$. So it would have been the same as with binary $p^\prime=1x\cdots$ and $q^\prime=1y\cdots$ with the $1$s occupying the $(N+1)$st bit from the right (representing $2^N$) and the two ellipses the same as the rightmost ellipses of $p$ and $q$. Thus $\omega(p,q)=\omega(p^\prime,q^\prime)$ where $p^\prime=2^N+p\bmod{2^N}$ and $q^\prime=2^N+q\bmod{2^N}$.
\end{proof}
\begin{example}
 $483\oplus481=2$ and $2^1\le2<2^2$ so $N=2$. So one concludes that $\omega(483,481)=\omega(2^2+483\bmod{2^2},2^2+481\bmod{2^2})=\omega(7,5)$ \demo
\end{example} 

\section{Conclusion}

The lack of a standard indexing system for the Cayley-Dickson basis vectors together with the fact that there is more than one possible doubling product have tended to obscure the basic periodicity of the twisting maps of these products. It is hoped that these modest results will help to clarify the issue and prove useful for further research.

\end{document}